\documentclass[final,1p,times]{elsarticle}
\biboptions{sort&compress}

\usepackage{amsmath,amssymb,amsthm,mathtools}
\usepackage{bm}
\usepackage{microtype}
\usepackage{enumitem}
\usepackage[hidelinks]{hyperref}
\usepackage{graphicx}
\usepackage{float}

\newtheorem{theorem}{Theorem}[section]
\newtheorem{lemma}[theorem]{Lemma}
\newtheorem{proposition}[theorem]{Proposition}
\newtheorem{corollary}[theorem]{Corollary}
\theoremstyle{remark}
\newtheorem{remark}[theorem]{Remark}
\newtheorem{assumption}[theorem]{Assumption}

\newcommand{\R}{\mathbb R}
\newcommand{\T}{^{\mathsf T}}
\newcommand{\dd}{\,\mathrm d}
\newcommand{\Ocal}{\mathcal O}
\newcommand{\jump}[1]{\mathopen{}\left[\!\left[#1\right]\!\right]\mathclose{}}

\begin{document}
\begin{frontmatter}
\title{Shock Structure for Hyperbolic Balance Laws \\ with Evolutionary Constraints}

\author[messina]{Natale Manganaro}
\ead{natale.manganaro@unime.it}
\address[messina]{Department of Mathematical and Computer Sciences, Physical Sciences and Earth Sciences, University of Messina, Messina, Italy}

\author[bologna,lincei]{Tommaso Ruggeri\corref{cor1}}
\ead{tommaso.ruggeri@unibo.it}
\cortext[cor1]{Corresponding author.}
\address[bologna]{Department of Mathematics and Alma Mater Research Center on Applied Mathematics (AM$^2$), Alma Mater Studiorum - University of Bologna, Bologna, Italy}
\address[lincei]{Accademia Nazionale dei Lincei, Rome, Italy}

\begin{abstract}
We study travelling shock profiles for dissipative quasilinear hyperbolic balance laws endowed with a convex entropy principle and subject to involutive differential constraints.  For systems without constraints, the classical Boillat--Ruggeri result excludes continuous shock profiles whose speed exceeds the largest characteristic velocity at the equilibrium state ahead of the shock.  We show that this bound remains unchanged for homogeneous constraints and for nonhomogeneous constraints whose source Jacobian has full row rank at equilibrium; a longitudinal two-species Gaussian ten-moment plasma with Landau collisions and Gauss's law provides a physical example of the latter case.  For rank-deficient constraints, travelling-wave compatibility alone leads, after restriction to the joint conservation--constraint manifold, to a quadratic correction determined by a finite-dimensional Lyapunov equation.  For $2\times2$ systems with one conservation law, one genuine balance law and one scalar involutive constraint, we prove that a regular noncharacteristic rank-deficient tail cannot sustain a nonzero correction, so that the original Boillat--Ruggeri bound is recovered.  The remaining degenerate case is illustrated by an exact propagated two-field model in which the constraint selects the travelling speed and the smooth profile approaches equilibrium algebraically.  The same model also admits entropy-admissible Lax composite profiles containing a sub-shock, while a completely smooth connection exists between the same equilibrium states. Numerical entropy evolutions further show that, depending on the smooth compatible initial datum, the first-order model may either remain smooth or dynamically develop a sub-shock.
\end{abstract}

\end{frontmatter}

\begin{center}
\emph{Dedicated with great affection to Guy Boillat, whose fundamental contributions have deeply shaped these subjects, wherever he may now be.}
\end{center}

\section{Introduction}

For a dissipative hyperbolic system of balance laws endowed with a convex
entropy, Boillat and Ruggeri \cite{BoillatRuggeri1998} proved that a regular
shock structure cannot propagate faster than the largest characteristic
velocity evaluated at the equilibrium state ahead of the shock.  The proof starts from the entropy principle, which must be satisfied by every
regular solution.  If the shock speed is larger than the largest characteristic
velocity at the equilibrium state ahead of the shock, the sign of the quadratic
term in the entropy inequality is reversed.  Hence the inequality cannot be
satisfied by a nontrivial smooth profile, and a continuous shock structure is
impossible.

This result has a particularly transparent interpretation in Rational Extended
Thermodynamics (RET).  The RET description enlarges the classical fields by
introducing nonequilibrium variables and their balance laws, thereby producing
a hierarchy of hyperbolic systems with finite characteristic speeds.  In moment
theories the largest characteristic velocity at equilibrium increases as more
moments are retained; in the classical hierarchy it grows without bound,
whereas in the relativistic hierarchy it remains below the speed of light and
approaches it as the number of moments increases
\cite{BoillatRuggeri1997Moments,BoillatRuggeri1999Rel}.

The development of the shock-structure theory considered here has a somewhat
unusual history.  Ruggeri \cite{Ruggeri1993Breakdown} identified the
coincidence of the shock speed with a characteristic velocity as a possible
mechanism for loss of regularity of a continuous travelling profile in a class
of dissipative hyperbolic systems whose production can be reduced to a
conservative block and a dissipative balance-law block.  This structural
setting is substantially more general than RET, which is a specific
nonequilibrium theory for gases and other continua.  RET nevertheless provided
the first important concrete testing ground for the mechanism.  Numerical
calculations by Weiss for the 13-, 14-, 20- and 21-moment equations closed
within RET revealed a remarkable regularization: singularities associated with
intermediate characteristic velocities can be crossed because numerator and
denominator vanish simultaneously along the profile, whereas genuine breakdown
is observed only at the largest equilibrium characteristic velocity
\cite{Weiss1995,Weiss1998Shock}.  This striking behavior was one of the
motivations for the subsequent Boillat--Ruggeri theorem stated above, which
again applies to the broader conservative--dissipative block structure and is
therefore not restricted to RET.

An earlier indication that the maximum equilibrium characteristic velocity
need not mark the onset of every sub-shock was obtained by Conforto, Mentrelli
and Ruggeri for a binary mixture of multi-temperature Eulerian fluids.  They
identified Mach-number ranges below the maximum unperturbed characteristic
velocity in which one or both constituents may develop a sub-shock inside the
shock structure \cite{ConfortoMentrelliRuggeri2017}.  Later, using the inverse
construction introduced by Mentrelli and Ruggeri \cite{MentrelliRuggeri2006},
Taniguchi and Ruggeri isolated the same phenomenon in simple dissipative
hyperbolic models in which a sub-shock occurs already below the largest
equilibrium characteristic velocity
\cite{TaniguchiRuggeri2018,TaniguchiRuggeri2019}.  These latter models
nevertheless possess a convex entropy, satisfy the entropy principle and also
the Shizuta--Kawashima (K-)condition \cite{ShizutaKawashima1985}.  Thus these structural properties alone
do not explain the regularization of the intermediate characteristic
singularities observed in genuine RET moment systems.  The additional
structural mechanism responsible for this behavior remains, to our knowledge,
open.  Related sub-shock phenomena have subsequently been investigated both
in RET models
\cite{TaniguchiArimaRuggeriSugiyama2014,RuggeriTaniguchi2022,RuggeriTaniguchi2024,ArimaRuggeriTaniguchi2026}
and in Grad-type moment systems for gas mixtures
\cite{BisiConfortoMartalo2016,ArtaleConfortoMartaloRicciardello2019,BisiGroppiMacalusoMartalo2021,ArtaleConfortoMartaloRicciardello2022}.
In a complementary direction, differential-constraint reductions for
nonhomogeneous hyperbolic systems and their wave solutions have been developed
by Manganaro and coauthors
\cite{ManganaroRizzo2025,JannelliManganaroRizzo2026}; continuous shock
structures and sub-shocks for a nonhomogeneous $p$-system were recently studied
in \cite{ManganaroRizzo2026PSystem}.

The purpose of the present paper is to understand what remains of the
Boillat--Ruggeri conclusion when the field equations are supplemented by
involutive differential constraints.  The mathematical theory of hyperbolic
systems with such constraints was developed in a series of works by Boillat
and Dafermos \cite{Boillat1982,Dafermos1986,Boillat1988}; for a broad account
of Boillat's formulation, see also his C.I.M.E. lectures
\cite{BoillatCIME1996}.  Involutive constraints arise, for example, in
magnetohydrodynamics and plasma physics, including relativistic
magnetohydrodynamics \cite{BoillatRuggeri1989MHD,Davidson2001}, and in the
three-plus-one formulation of the Einstein equations
\cite{ChoquetBruhat2009,ChoquetBruhatRuggeri1983}.

Our first result is that the Boillat--Ruggeri upper-speed condition
is unchanged for homogeneous constraints and for nonhomogeneous constraints
whose source Jacobian has full row rank at equilibrium.  A longitudinal
two-species Gaussian ten-moment plasma with Landau collisions and Gauss's law
provides a physical illustration of the full-rank case.  In
Section~\ref{sec:plasma} we verify directly the symmetric-hyperbolic entropy
structure, propagation of Gauss's constraint and the associated entropy
principle.

The rank-deficient case requires more care.  If one asks only for a constrained
travelling profile, the conservative block must first be integrated and the
profile restricted to the joint conservation--constraint manifold.  On its
stable tangent space the quadratic constraint contribution is represented by a
finite-dimensional Lyapunov equation.  This gives a corrected local spatial
threshold which is a property of the constrained travelling dynamics rather
than a new characteristic velocity of the hyperbolic PDE.

The physical question is whether such a correction survives when the
constraint is genuinely propagated by the PDE.  For the minimal $2\times2$
situation covered by the conservative--dissipative block structure---one
conservation law, one genuine balance law and one scalar constraint---we prove
that it does not in the regular noncharacteristic rank-deficient regime.  Strong propagation forces the quadratic part of the
constraint source to vanish along the one-dimensional compatible stable
direction whenever the spatial decay rate is nonzero.  Consequently the
quadratic correction determined by this equation vanishes and the Boillat--Ruggeri bound is
recovered.  The only alternative is a degenerate tail for which the linear
spatial rate vanishes; this regime lies outside the regular stable-manifold reduction described above
and must be analysed directly.

An exact two-field model within this broader structural class illustrates this second possibility.  Its source
is reducible by a constant linear recombination to one conservation law and one
balance law, its differential constraint is propagated exactly by the PDE, and
its entropy production is nonpositive throughout the constitutive domain.  The
constraint selects the travelling speed $s=1$.  A smooth heteroclinic profile
exists, but its approach to the equilibrium state is algebraic rather than
exponential.  The same model also possesses entropy-admissible Lax composite
profiles.  Moreover, the constrained evolution reduces exactly to a scalar
entropy problem, so that the nonuniqueness of the travelling-wave boundary-value
problem does not contradict uniqueness of the Cauchy problem.  Numerical
evolutions from two different smooth compatible initial data show respectively
persistence of a smooth structure and dynamical formation of a sub-shock.

The resulting picture is therefore different from a simple replacement of
$\lambda_{\max}$ by a new universal speed.  The Boillat--Ruggeri local bound is
robust under large classes of involutive constraints, and in the minimal
propagated rank-deficient two-field case it remains unchanged for every regular
noncharacteristic tail.  What the constraints add is a global compatibility
problem: they can select the shock speed, remove otherwise admissible orbits,
and create characteristic barriers for particular branches even when the
classical equilibrium inequality is satisfied.

The paper is organized as follows.  We introduce the entropy formalism for involutive constraints and the conservative--dissipative block structure, derive the exact
finite-amplitude entropy inequality, and recover the classical bound in the
full-rank and homogeneous cases.  After the plasma example we develop the
general rank-deficient travelling reduction and the quadratic correction determined by the associated Lyapunov equation.  We
then impose full PDE propagation and prove the two-field rigidity result.
The final section gives an exact propagated two-field model exhibiting speed
selection, a degenerate algebraic tail, coexistence of smooth and composite
profiles, and their dynamical realization through the associated scalar entropy
evolution.

\section{Balance laws, constraints and entropy principle}

Let
\[
U=U(x^1,\ldots,x^k,t)\in\mathcal D\subset\R^N,\qquad U_0\in\operatorname{int}\mathcal D,
\qquad
x=(x^1,\ldots,x^k)\in\R^k,
\]
where $t$ denotes time and $x^i$ ($i=1,\ldots,k$) are the spatial coordinates.
We use the notation
\[
\partial_t:=\frac{\partial}{\partial t},
\qquad
\partial_i:=\frac{\partial}{\partial x^i},
\]
and adopt the Einstein summation convention over repeated spatial indices.
We consider the quasilinear system of balance laws
\begin{equation}
 \partial_t F(U)+\partial_iF^i(U)=f(U),
 \label{eq:balance}
\end{equation}

We supplement \eqref{eq:balance} with $M<N$ involutive differential constraints
\begin{equation}
 \partial_iR^i(U)=q(U).
 \label{eq:constraint}
\end{equation}
Here
\[
F,F^i,f\in\R^N,\qquad
R^i,q\in\R^M.
\]
We restrict attention to genuinely dissipative balance laws, since purely
conservative systems do not provide the smooth shock structures considered
here.
Throughout the local analysis we assume that the constitutive maps entering
\eqref{eq:balance}--\eqref{eq:multipliers} are of class $C^3$ in a
neighborhood of the equilibrium states considered.  This regularity is used
only to justify the second-order Taylor expansions with cubic remainders.

In the mathematical theory of hyperbolic systems it is customary to use the
negative of the physical entropy density as the entropy function. Accordingly,
the mathematical entropy $h$ is convex and its production is written with the
sign $\Sigma\le0$, whereas the physical entropy is concave and its production
is nonnegative. In isothermal RET models the corresponding convex mathematical
entropy is typically represented by the total energy.
With this convention, assume that there exists a supplementary entropy law
\begin{equation}
 \partial_t h(U)+\partial_i h^i(U)=\Sigma(U)\le0.
 \label{eq:entropy}
\end{equation}
For the treatment of the entropy principle in the presence of involutive
constraints we follow Boillat \cite{Boillat1982,Boillat1988,BoillatCIME1996},
together with the main-field formalism of Ruggeri and Strumia
\cite{RuggeriStrumia1981}.  We assume the existence of a main field
$U'\in\R^N$ and a constraint multiplier $V\in\R^M$ such that
\begin{equation}
 dh=U'\cdot dF,
 \qquad
 dh^i=U'\cdot dF^i+V\cdot dR^i,
 \qquad
 \Sigma=U'\cdot f+V\cdot q\le0.
 \label{eq:multipliers}
\end{equation}

We assume that the entropy $h=h(F)$ is strictly convex with respect to the
densities $F$.  The first identity in \eqref{eq:multipliers} then gives
\[
 U'=\frac{\partial h}{\partial F},
\]
so that the map $F\mapsto U'$ is globally one-to-one on the state space under consideration and $U'$ may be used as dual variables on its image.  Introduce
the potentials
\begin{equation}
 h'=U'\cdot F-h,
 \qquad
 h'^i=U'\cdot F^i+V\cdot R^i-h^i.
 \label{eq:potentials}
\end{equation}
Thus $h'$ is the Legendre transform of $h$ with respect to $F$, with dual
variable $U'$, and
\begin{equation}
 F=\frac{\partial h'}{\partial U'},
 \qquad
 F^i=
 \frac{\partial h'^i}{\partial U'}
 -
 R^i\cdot\frac{\partial V}{\partial U'}.
 \label{eq:potentialrelations}
\end{equation}

From now on we use the main field as independent variable and write
\[
U\equiv U'.
\]
Throughout the paper, a comma followed by $U$ denotes differentiation with
respect to the main field. Thus $q_{,U}(U):\R^N\to\R^M$ is the Jacobian of
$q$, while $q_{,UU}(U)$ denotes its second derivative. The symbol $d$ is reserved for differentials with respect to the main-field
variables in identities such as \eqref{eq:multipliers}.
Derivatives with respect to an individual component are indicated in the same
way by the corresponding index after the comma.

Using \eqref{eq:potentialrelations} in the balance laws and the differential
constraint $\partial_iR^i=q$ to eliminate the terms containing
$\partial_iR^i$, we write the symmetric main-field system as
\[
 H(U)\,\partial_tU+A^i(U)\,\partial_iU
 =f(U)+V_{,U}(U)^{\mathsf T}q(U),
\]
where we have set
\begin{equation}
 H:=h'_{,UU},
 \qquad
 A^i:=h'^i_{,UU}-R^i\cdot V_{,UU}.
 \label{eq:HA}
\end{equation}
Since $h$ is strictly convex with respect to $F$, its Legendre transform $h'$
is strictly convex with respect to the main field $U$. Hence
\begin{equation}
 H(U)=h'_{,UU}(U)>0
 \label{eq:sym}
\end{equation}
throughout the constitutive domain under consideration.  The matrices $A^i$
are symmetric by construction, and the principal system is therefore
symmetric hyperbolic; in particular, for every unit direction $n$ the
generalized characteristic roots defined below are real.

For a unit vector $n=(n_i)$ define
\[
F_n=F^in_i,\qquad
R_n=R^in_i,\qquad
h_n=h^in_i,\qquad
h_n'=h'^in_i,\qquad
A_n=A^in_i.
\]
Subscripts preceding the comma retain their original meaning. For example,
$R_{n,U}$ and $h'_{n,UU}$ denote derivatives of $R_n$ and $h'_n$,
respectively.

The characteristic velocities of the constrained symmetric system are the
roots of
\begin{equation}
 \det(A_n-\lambda H)=0.
 \label{eq:char}
\end{equation}
We denote the largest one by $\lambda_{\max}(U,n)$.

\section{Conservative--dissipative block structure and equilibrium}
\label{sec:weakdiss}

The analysis below requires only that, after a constant nonsingular
recombination of the field equations, the production split into a conservative
block and a dissipative block of genuine balance laws.  The class of systems
considered here includes Rational Extended Thermodynamics (RET) as an important
physical subclass.  In RET, the conservative block typically contains mass,
momentum and total energy, while the additional nonequilibrium variables
satisfy balance laws on which the production acts.  The precise assumption used
below is the following.

\begin{assumption}[Conservative--dissipative block structure]
\label{ass:RET}
Up to a constant nonsingular recombination of the $N$ field equations, the
production term of system \eqref{eq:balance} has the block form
\begin{equation}
 f(U)=\binom{0}{g(U)},
 \qquad g(U)\in\mathbb R^{N-r}.
 \label{eq:RET-source}
\end{equation}
Accordingly, the first $r$ recombined equations are genuine conservation laws,
while the remaining $N-r$ equations are balance laws for the extended
variables.  We write the corresponding main field as $U=(v,w)$, with
$v\in\mathbb R^r$ and $w\in\mathbb R^{N-r}$.
\end{assumption}

Following M\"uller and Ruggeri \cite{MullerRuggeri1998}, we use the standard
thermodynamic definition of equilibrium.  \emph{At equilibrium the physical
entropy production vanishes and attains its minimum value.  Equivalently, with
the sign convention $\Sigma\le0$ adopted here, $\Sigma$ vanishes and
attains its maximum value, namely zero.}  In the classical unconstrained RET setting, the
strict dissipativity condition
\[
 D(v,w):=\frac12\left(g_{,w}+g_{,w}^{\mathsf T}\right)<0
\]
implies $w=0$ at equilibrium.  In the classical unconstrained setting, the
condition $w=0$ is also consistent with the notion of an equilibrium principal
subsystem \cite{BoillatRuggeri1997Principal}.  In the constrained setting
considered here we retain only this pointwise equilibrium property as an
explicit hypothesis.

\begin{assumption}[Equilibrium normalization]
\label{ass:equilibrium}
In the representation of Assumption~\ref{ass:RET}, the reference equilibrium
admits the normalization
\begin{equation}
 U_0=(v_0,0),
 \qquad f(U_0)=0,
 \qquad q(U_0)=0.
 \label{eq:RET-equilibrium}
\end{equation}
Equivalently, the nonequilibrium component of the main field vanishes at the
reference equilibrium.  This is the only equilibrium-subsystem property
required below.
\end{assumption}

The formulation is invariant under the constant recombination used to expose
this block structure.  Indeed, if the field equations are multiplied by an
invertible constant matrix $T$, then $F$, $F^i$, and $f$ are multiplied by $T$,
while the main field is replaced by $T^{-\mathsf T}U$, leaving all multiplier
identities invariant.

Assumptions~\ref{ass:RET} and \ref{ass:equilibrium} immediately yield
\begin{equation}
 U_0\cdot f(U)\equiv0.
 \label{eq:orthf}
\end{equation}
The entropy production in main-field variables is therefore
\begin{equation}
 \Sigma(U)=U\cdot f(U)+V(U)\cdot q(U)\le0,
 \label{eq:Sigma}
\end{equation}
which can be rewritten as
\begin{equation}
 \Sigma(U)=(U-U_0)\cdot f(U)+V(U)\cdot q(U)\le0.
 \label{eq:Sigma-TW-identity}
\end{equation}
The role of Assumption~\ref{ass:equilibrium} is precisely to ensure, together
with Assumption~\ref{ass:RET}, the orthogonality relation \eqref{eq:orthf},
which allows the entropy production \eqref{eq:Sigma} to be written in the
relative form \eqref{eq:Sigma-TW-identity}.  This equilibrium normalization is
satisfied by the physical and model examples considered below.

\begin{lemma}[Equilibrium consequence of the entropy principle]
\label{lem:weakstationarity}
Under Assumptions~\ref{ass:RET} and \ref{ass:equilibrium},
\begin{equation}
 V_0\cdot q_{,U}(U_0)=0,
 \qquad V_0:=V(U_0),
 \label{eq:VDq}
\end{equation}
and
\begin{equation}
 \Sigma_{,UU}(U_0)\le0
 \label{eq:weakHessian}
\end{equation}
as a quadratic form.
\end{lemma}

\begin{proof}
By \eqref{eq:RET-equilibrium}, $\Sigma(U_0)=0$.  Since
$\Sigma\le0$ in a neighborhood of $U_0$, the equilibrium is a local maximum,
so $\Sigma_{,U}(U_0)=0$ and $\Sigma_{,UU}(U_0)\le0$.  Differentiating
\eqref{eq:Sigma} at $U_0$ gives
\[
 \Sigma_{,U}(U_0)
 =U_0\cdot f_{,U}(U_0)+V_0\cdot q_{,U}(U_0).
\]
The first term vanishes identically by \eqref{eq:orthf}, and
\eqref{eq:VDq} follows.
\end{proof}

\section{Travelling shock structures}

A plane shock structure is a $C^1$ travelling wave
\begin{equation}
 U=U(\phi),
 \qquad
 \phi=n_ix_i-st,
 \label{eq:TW}
\end{equation}
satisfying
\begin{equation}
 \lim_{\phi\to-\infty}U(\phi)=U_1,
 \qquad
 \lim_{\phi\to+\infty}U(\phi)=U_0,
 \qquad
 \lim_{\phi\to\pm\infty}\dot U(\phi)=0,
 \label{eq:bc}
\end{equation}
where $U_0$ is the unperturbed equilibrium state ahead of the shock and $U_1$
the state behind it.
For piecewise-smooth profiles we denote the jump of a quantity $a$ across a
discontinuity by $\jump{a}:=a^+-a^-$.

The balance laws \eqref{eq:balance} reduce to
\begin{equation}
 \frac{\dd}{\dd\phi}
 \{-sF(U)+F_n(U)\}
 =
 f(U),
 \label{eq:TWbalance}
\end{equation}
while the involutive constraints \eqref{eq:constraint} become
\begin{equation}
 \frac{\dd}{\dd\phi}R_n(U)=q(U).
 \label{eq:TWconstraint}
\end{equation}

We denote differentiation with respect to the travelling-wave variable $\phi$ by a dot,
\[
 \dot U:=\frac{\dd U}{\dd\phi}.
\]
Along a travelling wave, the differential constraints therefore read
\begin{equation}
 R_{n,U}(U)\,\dot U=q(U).
 \label{eq:TWconstraint-differential}
\end{equation}
Hence $R_n(U)$ is constant along the profile only in the homogeneous case $q\equiv0$; for nonhomogeneous constraints, the variation of $R_n(U)$ along the profile is determined by the source $q(U)$.

\section{Exact constrained entropy inequality}

Multiply \eqref{eq:TWbalance} by $U-U_0$.  Using the multiplier identities
\eqref{eq:multipliers}, the travelling constraint
\eqref{eq:TWconstraint}, and the identity \eqref{eq:Sigma-TW-identity}, one obtains
\begin{align}
 &\frac{\dd}{\dd\phi}
 \left\{
 -sh(U)+h_n(U)
 -
 U_0\cdot[-sF(U)+F_n(U)]
 \right\}
 \nonumber\\
 &\qquad
 =
 (U-U_0)\cdot f(U)+V(U)\cdot q(U)
 \le0.
 \label{eq:key}
\end{align}

Since $U(\phi)\to U_0$ as $\phi\to+\infty$, integrating
\eqref{eq:key} from a point $\phi$ of the profile to $+\infty$ gives
\begin{align}
 &-s[h(U)-h(U_0)]
 +h_n(U)-h_n(U_0)
 \nonumber\\
 &\quad
 -U_0\cdot
 \{-s[F(U)-F(U_0)]
 +F_n(U)-F_n(U_0)\}
 \ge0.
 \label{eq:integrated}
\end{align}

Using the potentials \eqref{eq:potentials}, this is equivalent to
\begin{align}
 \mathcal E(U,U_0):={}&
 s[h'(U)-h'(U_0)]
 -[h_n'(U)-h_n'(U_0)]
 \nonumber\\
 &+(U-U_0)\cdot
 \left[
 -s h'_{,U}(U)+h'_{n,U}(U)
 \right]
 \nonumber\\
 &+V(U)\cdot R_n(U)
 -V(U_0)\cdot R_n(U_0)
 \nonumber\\
 &-(U-U_0)\cdot
 \left[
 R_n(U)\cdot V_{,U}(U)
 \right]
 \ge0.
 \label{eq:E}
\end{align}

Equation \eqref{eq:E} is the constrained finite-amplitude analogue of the
basic inequality used in the unconstrained Boillat--Ruggeri theorem.

\section{Local expansion at the equilibrium state}

Set
\[
 X:=U-U_0,
 \qquad
 V_0:=V(U_0),
 \qquad
 R_0:=R_n(U_0).
\]

We expand
\begin{align}
 V(U)
 &=V_0+V_{,U}(U_0)X
   +\frac12V_{,UU}(U_0)[X,X]+\Ocal(|X|^3),
 \label{eq:Vexp}
 \\
 R_n(U)
 &=R_0+R_{n,U}(U_0)X
   +\frac12R_{n,UU}(U_0)[X,X]+\Ocal(|X|^3).
 \label{eq:Rexp}
\end{align}

The part of \eqref{eq:E} involving $h'$ and $h_n'$ gives
\begin{equation}
 \frac12X\T
 [h'_{n,UU}(U_0)-sH(U_0)]
 X
 +\Ocal(|X|^3).
 \label{eq:hpart}
\end{equation}

The constraint part satisfies
\begin{align}
 &V(U)\cdot R_n(U)-V_0\cdot R_0
 -X\cdot[R_n(U)\cdot V_{,U}(U)]
 \nonumber\\
 &\qquad
 =V_0\cdot R_{n,U}(U_0)X
 +\frac12V_0\cdot R_{n,UU}(U_0)[X,X]
 \nonumber\\
 &\qquad\phantom{=}
 -\frac12R_0\cdot V_{,UU}(U_0)[X,X]
 +\Ocal(|X|^3).
 \label{eq:constraintpart}
\end{align}
The mixed quadratic terms containing $V_{,U}(U_0)$ and $R_{n,U}(U_0)$ cancel
exactly.

Since
\begin{equation}
 A_n(U_0)
 =h'_{n,UU}(U_0)-R_0\cdot V_{,UU}(U_0),
 \label{eq:A0}
\end{equation}
we arrive at the basic expansion
\begin{equation}
 \begin{aligned}
 \mathcal E(U,U_0)
 &=V_0\cdot R_{n,U}(U_0)X \\
 &\quad+\frac12 X\T
 \left[A_n(U_0)-sH(U_0)+V_0\cdot R_{n,UU}(U_0)\right]X \\
 &\quad+\Ocal(|X|^3)\ge0.
 \end{aligned}
 \label{eq:master}
\end{equation}

This formula is the central local identity for the constrained
shock-structure problem.

There is a more useful equivalent form along an actual travelling profile.
Define
\begin{equation}
 \rho(\phi):=V_0\cdot\bigl[R_n(U(\phi))-R_n(U_0)\bigr].
 \label{eq:rho}
\end{equation}
Taylor expansion gives
\[
 \rho
 =V_0\cdot R_{n,U}(U_0)X
 +\frac12V_0\cdot R_{n,UU}(U_0)[X,X]
 +\Ocal(|X|^3).
\]
Hence \eqref{eq:master} is equivalent, along the profile, to
\begin{equation}
 \mathcal E(U,U_0)
 =
 \rho(\phi)
 +\frac12X\T[A_n(U_0)-sH(U_0)]X
 +\Ocal(|X|^3)\ge0.
 \label{eq:profilemaster}
\end{equation}
Thus the curvature $R_{n,UU}(U_0)$ does not produce an independent quadratic
term in the speed condition: it is already contained in $\rho$.  Moreover, the travelling
constraint implies the exact identity
\begin{equation}
 \rho(\phi)
 =-
 \int_{\phi}^{+\infty}V_0\cdot q(U(\psi))\,\dd\psi,
 \label{eq:rhoexact}
\end{equation}
whenever the integral converges.  Formulae \eqref{eq:profilemaster} and
\eqref{eq:rhoexact} are the convenient starting point for all three cases
below.

\section{Nonhomogeneous constraints with full-rank source Jacobian}
\label{sec:nondegenerate}

Lemma~\ref{lem:weakstationarity} gives the equilibrium stationarity relation
\[
 V_0\cdot q_{,U}(U_0)=0.
\]
If the constraint-source Jacobian has full rank at equilibrium, this already forces the
constraint multiplier to vanish.

\begin{proposition}[Vanishing of the equilibrium constraint multiplier]
Assume
\begin{equation}
 \operatorname{rank}q_{,U}(U_0)=M.
 \label{eq:rank}
\end{equation}
Then
\begin{equation}
 V_0=0.
 \label{eq:V0zero}
\end{equation}
\end{proposition}

\begin{proof}
The map $q_{,U}(U_0):\R^N\to\R^M$ has full row rank.  Hence its left kernel is
trivial, and \eqref{eq:VDq} implies $V_0=0$.
\end{proof}

With $V_0=0$, the master expansion \eqref{eq:master} reduces to
\begin{equation}
 \mathcal E(U,U_0)
 =
 \frac12
 X\T[A_n(U_0)-sH(U_0)]X
 +
 \Ocal(|X|^3)\ge0.
 \label{eq:reducedA}
\end{equation}

\begin{theorem}[Nonhomogeneous constraints with full-rank source Jacobian]
\label{thm:nondegenerate}
Let Assumptions~\ref{ass:RET} and \ref{ass:equilibrium} hold, let $H(U_0)>0$, and assume
\begin{equation}
 \operatorname{rank}q_{,U}(U_0)=M.
\end{equation}
Then a nontrivial $C^1$ shock structure approaching $U_0$ can exist only if
\begin{equation}
 s\le\lambda_{\max}(U_0,n).
 \label{eq:boundA}
\end{equation}
\end{theorem}

\begin{proof}
By Lemma~\ref{lem:weakstationarity} and the previous proposition,
$V_0=0$, so \eqref{eq:reducedA} holds.  Assume, by contradiction,
\[
s>\lambda_{\max}(U_0,n).
\]
Since $s>\lambda_{\max}(U_0,n)$, the symmetric matrix
$A_n(U_0)-sH(U_0)$ is negative definite.  Hence the leading quadratic term
in \eqref{eq:reducedA} is negative for every nonzero $X$, and therefore
$\mathcal E(U,U_0)<0$ sufficiently close to $U_0$ along any nontrivial
profile.  This contradicts the entropy inequality $\mathcal E\ge0$.  Hence
$s\le\lambda_{\max}(U_0,n)$.
\end{proof}

\section{A one-dimensional dissipative plasma example}
\label{sec:plasma}

We now verify in some detail that Theorem~\ref{thm:nondegenerate} applies to a
standard ten-moment plasma closure.  The model may be viewed as a finite-moment
reduction of a mesoscopic kinetic description: for each species the
distribution function satisfies a Vlasov--Fokker--Planck--Landau equation,
coupled with Maxwell's equations, and taking velocity moments generates an
infinite hierarchy of balance equations.  The hierarchy is closed here at the
level of the second velocity moment by the maximum-entropy principle (MEP).
A closely related derivation of extended plasma moment systems directly from
the Fokker--Planck--Landau equation coupled with Maxwell's equations was given
by Al\`i, Mascali, Pezzi and Valentini \cite{AliMascaliPezziValentini2023}.

The MEP has been used in many areas of statistical mechanics and continuum
modelling.  In the context of moment equations for the Boltzmann equation, its
systematic use as a closure procedure goes back to Kogan.  M\"uller and Ruggeri
subsequently showed that, whenever the relevant moments are finite, the fully
nonlinear MEP closure yields a symmetric-hyperbolic system in which the
Lagrange multipliers constitute the main field.  Levermore later developed a
rigorous entropy-based hierarchy of moment closures and analysed their
admissibility
\cite{Kogan1969,MullerRuggeri1993,Levermore1996,ArimaRuggeri2025MEP}.
At the ten-moment level, maximization of the kinetic entropy under prescribed
density, momentum and symmetric second moment yields an anisotropic Gaussian.
This Gaussian closure is fully nonlinear and nonperturbative, and its convex
entropy structure is defined throughout the admissible realizability domain;
see also the detailed ten-moment analysis in
\cite{LevermoreMorokoff1998}.

This point is relevant for the present application.  The explicit
thirteen-moment plasma closure of Al\`i et al. is obtained by expanding the MEP
distribution to first order in the anisotropy, an approximation that reduces
the hyperbolicity region.  We therefore retain here the fully nonlinear
Gaussian ten-moment closure, which does not require a near-equilibrium
expansion.

Let $a\in\{e,i\}$ label electrons and ions.  For each species, $n_a$ is the
number density, $m_a$ the particle mass, $q_a$ the electric charge,
$\boldsymbol u_a\in\mathbb R^3$ the macroscopic velocity, and $\Theta_a>0$
the symmetric velocity-covariance tensor.  The pressure tensor is
$P_a=m_an_a\Theta_a$.  With $\boldsymbol{\xi}\in\mathbb R^3$ denoting
molecular velocity, the corresponding maximum-entropy distribution is
\begin{equation}
 \begin{aligned}
 f_a^G(\boldsymbol{\xi})
 &=\frac{n_a}{(2\pi)^{3/2}\sqrt{\det\Theta_a}}
 \exp\!\left[-\frac12(\boldsymbol{\xi}-\boldsymbol u_a)\T
 \Theta_a^{-1}(\boldsymbol{\xi}-\boldsymbol u_a)\right],\\
 P_a&=m_an_a\Theta_a>0.
 \end{aligned}
 \label{eq:gaussian}
\end{equation}

For the Gaussian distribution \eqref{eq:gaussian}, this identification can be
seen directly from
\begin{equation}
 \log\frac{f_a^G}{y_a}
 =\alpha_a+\boldsymbol\beta_a\cdot\boldsymbol\xi
   +\boldsymbol\xi\cdot\Gamma_a\boldsymbol\xi,
 \qquad
 \boldsymbol\beta_a=\Theta_a^{-1}\boldsymbol u_a,
 \qquad
 \Gamma_a=-\frac12\Theta_a^{-1},
 \label{eq:plasmaMultipliers}
\end{equation}
where
\[
 \alpha_a=\log\!\frac{n_a}{y_a(2\pi)^{3/2}\sqrt{\det\Theta_a}}
 -\frac12\boldsymbol u_a\cdot\Theta_a^{-1}\boldsymbol u_a .
\]
Thus, apart from the fixed normalizations determined by the chosen moment
variables, $(\alpha_a,\boldsymbol\beta_a,\Gamma_a)$ are the entropy
multipliers, hence the components of the main field for the ten retained
moments.  For the plasma application considered here, the collisional source
is not taken from the gas-dynamic model: it is obtained by evaluating the
multispecies Landau operator on the Gaussians and retaining the corresponding
ten moments \cite{Pfefferle2017}.

For the one-dimensional longitudinal reduction we write
$\boldsymbol u_a=(u_a,0,0)$ and
$P_a=\operatorname{diag}(P_{\parallel a},P_{\perp a},P_{\perp a})$, where
$P_{\parallel a}$ and $P_{\perp a}$ are respectively the longitudinal and
transverse pressures.  The electric field is $\boldsymbol E=(E,0,0)$, the
magnetic field is set to $\boldsymbol B=0$, $\varepsilon$ denotes the electric
permittivity, and
$J=\sum_a q_an_au_a$ is the longitudinal electric current density.  Ampere's
equation and Gauss's law are
\begin{equation}
 \varepsilon E_t=-J,\qquad
 J=\sum_aq_an_au_a,
 \qquad
 E_x=\mathcal Q(U):=\frac1\varepsilon\sum_aq_an_a.
 \label{eq:plasmaMaxwell}
\end{equation}
The involutive character of Gauss's law is explicit.  Indeed, the species
continuity equations imply charge conservation,
\begin{equation}
 \rho_{c,t}+J_x=0,
 \qquad \rho_c:=\sum_aq_an_a.
 \label{eq:chargeConservation}
\end{equation}
Hence, for the constraint defect
\begin{equation}
 \mathcal C:=E_x-\frac{\rho_c}{\varepsilon},
 \label{eq:GaussDefect}
\end{equation}
Ampere's equation and \eqref{eq:chargeConservation} give
\begin{equation}
 \mathcal C_t
 =E_{xt}-\frac{\rho_{c,t}}{\varepsilon}
 =-\frac{J_x}{\varepsilon}+\frac{J_x}{\varepsilon}=0.
 \label{eq:GaussPropagation}
\end{equation}
Thus Gauss's constraint is not merely compatible with a travelling reduction:
it is propagated exactly by every smooth solution of the full PDE.

Let $M_a$ denote the vector of the ten retained moments (density, the three
momentum components and the six independent components of the symmetric
second moment), and let $S_a^L$ be their collisional production generated by
the Landau operator.  The convex Gaussian kinetic entropy is
\begin{equation}
 \eta_a(M_a)
 =k_B\int_{\mathbb R^3} f_a^G
 \left(\log\frac{f_a^G}{y_a}-1\right)\,d\boldsymbol{\xi},
 \qquad \eta_{a,M_aM_a}(M_a)>0.
 \label{eq:gaussianentropy}
\end{equation}
Because $\log f_a^G$ is exactly quadratic in $\boldsymbol\xi$, contraction of
the retained Landau moments with the entropy variables
$\eta_{a,M_a}(M_a)$ reproduces the kinetic entropy production without any
truncation remainder,
\begin{equation}
 \sum_a\eta_{a,M_a}(M_a)\cdot S_a^L=-\sigma_L\le0,
 \label{eq:LandauEntropy}
\end{equation}
where $\sigma_L\ge0$ is the total Landau entropy dissipation.  The Landau
operator conserves each species mass, total momentum and total kinetic energy
\cite{GualdaniZamponi2017}.  In general, a kinetic $H$-theorem need not descend
to an entropy principle for an arbitrary finite moment truncation
\cite{BaeHaHwangRuggeri2024}; here it does precisely because the entropy
multiplier belongs to the retained quadratic moment space.  The electric force
term does not contribute to the kinetic entropy balance, since it is a velocity-space divergence.

Choose a neutral common Maxwellian equilibrium, with $T_0>0$ denoting the
common equilibrium temperature and the subscript $0$ denoting equilibrium
values,
\begin{equation}
 u_{a0}=0,\qquad
 P_{\parallel a,0}=P_{\perp a,0}=n_{a0}k_BT_0,\qquad
 E_0=0,\qquad \sum_aq_an_{a0}=0.
 \label{eq:plasmaeq}
\end{equation}
Subtracting the equilibrium linear invariants from the material kinetic
entropy and using the conserved total energy, including the electric-field
energy, gives the strictly convex relative free-energy entropy
\begin{equation}
 h=\sum_a k_B\int_{\mathbb R^3}
 \left[
 f_a^G\log\frac{f_a^G}{f_{a0}}
 -f_a^G+f_{a0}
 \right]d\boldsymbol\xi
 +\frac{\varepsilon E^2}{2T_0},
 \label{eq:plasmaRelativeEntropy}
\end{equation}
up to inessential linear combinations of the conserved species masses.
Denoting by $h^x$ the corresponding relative entropy flux, the exact entropy
balance reads
\begin{equation}
 \partial_th+\partial_xh^x=-\sigma_L\le0.
 \label{eq:plasmaentropy}
\end{equation}
No use of Gauss's constraint is made in deriving
\eqref{eq:plasmaentropy}.  Consequently the constraint multiplier appearing
in the constrained entropy identity can be chosen identically zero,
\begin{equation}
 V\equiv0.
 \label{eq:plasmaVzero}
\end{equation}
The collision invariants supply the conservative block required by Assumption~\ref{ass:RET}, and the
nonconserved moments form the balance block.  In the corresponding relative
entropy variables the nonequilibrium component of the equilibrium main field
vanishes, in agreement with the equilibrium characterization stated in
Section~\ref{sec:weakdiss}.

The source Jacobian in Gauss's law has full rank.  In the physical density
variables its nonzero part is
\begin{equation}
 \frac1\varepsilon(q_e,q_i),
 \qquad \operatorname{rank}\mathcal Q_{,U}(U_0)=1=M.
 \label{eq:plasmaRank}
\end{equation}
Here $M=1$ because Gauss's law is the single scalar involutive constraint.
The rank is unchanged by the locally invertible transformation to entropy
variables.  If $\lambda$ denotes a characteristic speed in the longitudinal
$x$-direction, the ten-moment block of each species has characteristic
polynomial
\[
 (\lambda-u_a)^2
 \left[(\lambda-u_a)^2-\frac{3P_{\parallel a}}{m_an_a}\right],
\]
while Ampere's equation adds a zero speed.  The factor $3$ is not the equilibrium specific-heat ratio: it arises from the homogeneous longitudinal ten-moment pressure equation, whose quasilinear part contains $3P_{\parallel a}\,\partial_x u_a$.  In the isotropic Euler limit the corresponding acoustic coefficient is instead $\gamma_a p_a/\rho_a$, with $\gamma_a=5/3$ for a monatomic species.  Hence at \eqref{eq:plasmaeq}
\begin{equation}
 \lambda_{\max}(U_0)=\max_{a=e,i}\sqrt{\frac{3k_BT_0}{m_a}}.
 \label{eq:plasmaLambda}
\end{equation}

\begin{theorem}[Explicit one-dimensional plasma speed bound]
\label{thm:plasma}
Any nontrivial $C^1$ travelling shock profile of the longitudinal two-species
Gaussian ten-moment system with Landau collisions, Ampere evolution and Gauss
constraint, approaching \eqref{eq:plasmaeq} ahead of the shock, must satisfy
\begin{equation}
 \boxed{
 s\le\max_{a=e,i}\sqrt{\frac{3k_BT_0}{m_a}}.
 }
 \label{eq:plasmaBound}
\end{equation}
For an ordinary electron--ion plasma, $m_e<m_i$, and the right-hand side is the
electron value.
\end{theorem}

\begin{proof}
Strict convexity and the entropy inequality follow from the relative Gaussian
free energy and \eqref{eq:plasmaentropy}.  Equation~\eqref{eq:GaussPropagation}
verifies full PDE propagation of the constraint, while
\eqref{eq:plasmaVzero} shows that its entropy multiplier is zero.  The
collision-invariant splitting is precisely Assumption~\ref{ass:RET}, and
\eqref{eq:plasmaRank} is the full-row-rank condition.  Apply
Theorem~\ref{thm:nondegenerate} and use \eqref{eq:plasmaLambda}.
\end{proof}

This physical example is therefore complementary to the rank-deficient
construction below: Gauss's law is a genuine propagated involutive constraint,
its source Jacobian has full row rank at equilibrium, the entropy principle
holds for the closed ten-moment--Landau system, and $V\equiv0$.

\section{Homogeneous constraints}

If $q\equiv0$, the travelling constraint gives
$R_n(U(\phi))=R_n(U_0)$ along the whole profile.  Expanding this identity at
$U_0$ and multiplying by $V_0$ cancels, to quadratic order, the two
$V_0$-dependent terms in the master expansion.  Hence
\begin{equation}
 \mathcal E(U,U_0)
 =\frac12X\T[A_n(U_0)-sH(U_0)]X+\Ocal(|X|^3)\ge0,
 \label{eq:reducedB}
\end{equation}
without any assumption on $V_0$.

\begin{theorem}[Homogeneous constraints]
\label{thm:homogeneous}
Let Assumptions~\ref{ass:RET} and \ref{ass:equilibrium} hold.  If $q\equiv0$ and $H(U_0)>0$, then a
nontrivial $C^1$ shock structure approaching $U_0$ can exist only if
\begin{equation}
 s\le\lambda_{\max}(U_0,n).
 \label{eq:boundB}
\end{equation}
\end{theorem}

\begin{proof}
Use \eqref{eq:reducedB} and repeat the definiteness argument of
Theorem~\ref{thm:nondegenerate}.
\end{proof}

The homogeneous case includes propagated divergence-type constraints and is
closely related, in the conservative setting, to Dafermos's involutions
\cite{Dafermos1986}.  The genuinely dissipative nonhomogeneous case treated
below is different because the multiplier contribution need not cancel.

\section{Nonhomogeneous constraints with rank-deficient source Jacobian}
\label{sec:degenerate}

Assumptions~\ref{ass:RET} and \ref{ass:equilibrium} remain in force.  We now consider
\begin{equation}
 q\not\equiv0,
 \qquad
 \operatorname{rank}q_{,U}(U_0)<M,
 \qquad
 V_0\ne0.
 \label{eq:degenerate}
\end{equation}
Lemma~\ref{lem:weakstationarity} gives only
\begin{equation}
 V_0\cdot q_{,U}(U_0)=0,
 \label{eq:VDqdeg}
\end{equation}
and the equilibrium multiplier need not vanish.  The constraint contribution
to the relative entropy inequality must therefore be determined from the
actual asymptotic travelling dynamics.

\subsection{Integral form of the constraint contribution}

We now use the identities \eqref{eq:rho} and \eqref{eq:rhoexact}.
Since $q(U_0)=0$ and \eqref{eq:VDqdeg} holds,
\begin{equation}
 V_0\cdot q(U_0+X)
 =\frac12X\T QX+\Ocal(|X|^3),
 \qquad
 Q:=V_0\cdot q_{,UU}(U_0)=Q\T.
 \label{eq:Qdef}
\end{equation}
Hence
\begin{equation}
 \rho(\phi)
 =-\frac12\int_\phi^{+\infty}X(\psi)\T QX(\psi)\,\dd\psi
 +\Ocal\!\left(\int_\phi^{+\infty}|X(\psi)|^3\,\dd\psi\right).
 \label{eq:rhoQ}
\end{equation}
Thus the leading constraint contribution is quadratic and will be determined
below through a finite-dimensional Lyapunov equation.

\subsection{Reduced travelling dynamics and the conservation leaf}

Differentiating the potential relations \eqref{eq:potentialrelations} gives
\begin{equation}
 F_{n,U}=A_n-V_{,U}\T R_{n,U},
 \qquad
 F_{,U}=H.
 \label{eq:FnU}
\end{equation}
Combining the travelling balance law \eqref{eq:TWbalance} with the travelling
constraint \eqref{eq:TWconstraint} yields
\begin{equation}
 [A_n(U)-sH(U)]\dot U
 =f(U)+V_{,U}(U)\T q(U).
 \label{eq:reducedTW}
\end{equation}
Assume that the shock speed is noncharacteristic at the reference equilibrium,
so that
\begin{equation}
 B_s:=A_n(U_0)-sH(U_0)
 \label{eq:Bs}
\end{equation}
is invertible.  By continuity, $A_n(U)-sH(U)$ is then invertible in a
neighborhood of $U_0$, and every admissible profile satisfies there
\begin{equation}
 \dot U=\mathcal G_s(U)
 :=[A_n(U)-sH(U)]^{-1}
 [f(U)+V_{,U}(U)\T q(U)],
 \label{eq:Gs}
\end{equation}
Linearizing \eqref{eq:Gs} about the equilibrium state $U_0$, with
$X=U-U_0$, gives
\begin{equation}
 \dot X=L_sX+\Ocal(|X|^2),
 \qquad
 L_s:=B_s^{-1}
 \left[f_{,U}(U_0)+V_{,U}(U_0)\T q_{,U}(U_0)\right].
 \label{eq:Ls}
\end{equation}

The reduced dynamics \eqref{eq:Gs} must also satisfy the travelling constraint
\eqref{eq:TWconstraint}.  Substituting $\dot U=\mathcal G_s(U)$ into the latter
leads to the compatibility condition
\begin{equation}
 \Gamma_s(U):=R_{n,U}(U)\mathcal G_s(U)-q(U)=0.
 \label{eq:Gamma}
\end{equation}
Thus an admissible travelling profile is confined to the set
$\Gamma_s(U)=0$.  This is not the whole reduction, because the conservative
block in Assumption~\ref{ass:RET} can be integrated exactly.
Writing $F^{\rm c}$ and $F_n^{\rm c}$ for the densities and normal fluxes of
the conservation laws in Assumption~\ref{ass:RET}, define
\begin{equation}
 \mathcal J_s(U):=F_n^{\rm c}(U)-sF^{\rm c}(U).
 \label{eq:Js}
\end{equation}
Every travelling profile approaching $U_0$ satisfies
\begin{equation}
 \mathcal J_s(U)=\mathcal J_s(U_0).
 \label{eq:conservationleaf}
\end{equation}
This is the local Rankine--Hugoniot leaf associated with the conservative
variables.

We therefore introduce the combined travelling set
\begin{equation}
 \mathcal N_s:=\left\{U:\
 \mathcal J_s(U)=\mathcal J_s(U_0),\quad
 \Gamma_s(U)=0\right\}.
 \label{eq:Ns}
\end{equation}
Every smooth constrained travelling profile lies in $\mathcal N_s$.  The conservation leaf is essential: neutral directions caused by the
conserved quantities must be removed before the spectrum of the reduced
spatial linearization is analysed.

For the local theorem we impose only the following analytic regularity
conditions.  In a neighborhood of $U_0$, $\mathcal N_s$ is an embedded $C^2$
submanifold invariant under the reduced flow $\mathcal G_s$.  Moreover, $U_0$
is a hyperbolic equilibrium of the restricted flow
$\mathcal G_s|_{\mathcal N_s}$, namely, the linearization of this restricted
flow at $U_0$ has no eigenvalues with zero real part.  These are not additional
constitutive assumptions; they are the analytic regularity conditions needed
to apply the stable-manifold theorem to the reduced flow on $\mathcal N_s$.

Since $U_0$ is hyperbolic, the tangent space $T_{U_0}\mathcal N_s$ splits into
stable and unstable invariant subspaces.  We denote by
\begin{equation}
 E:=E^s\!\left(L_s|_{T_{U_0}\mathcal N_s}\right),
 \qquad d:=\dim E,
 \label{eq:Estable}
\end{equation}
the stable subspace, namely the span of the generalized eigendirections whose
eigenvalues have negative real part.  The local stable manifold through $U_0$
is tangent to $E$, and any travelling profile satisfying
$U(\phi)\to U_0$ as $\phi\to+\infty$ eventually lies on this manifold.

If $d=0$, there are no stable directions for the restricted linearized flow,
and therefore no nontrivial nearby admissible profile can approach $U_0$ as
$\phi\to+\infty$.  If $d\ge1$, choose a basis matrix
$S\in\mathbb R^{N\times d}$ whose range is the stable subspace $E$.  The
stable-manifold theorem then provides local coordinates $y\in\mathbb R^d$ on
the stable manifold, with $y=0$ corresponding to $U_0$, such that
\begin{align}
 X&=Sy+\Ocal(|y|^2),
 \label{eq:stablegraph}\\
 \dot y&=L_Ey+\Ocal(|y|^2),
 \qquad L_sS=SL_E,
 \qquad \max\Re\sigma(L_E)<0.
 \label{eq:stableflow}
\end{align}
Thus $L_E$ represents the linearized dynamics restricted to the stable
directions, and all its modes decay as $\phi\to+\infty$.

\begin{lemma}[Automatic quadratic-tail estimate]
\label{lem:ytail}
For every admissible profile described by \eqref{eq:stablegraph}--
\eqref{eq:stableflow}, there exists $C_E>0$ such that, for all sufficiently
large $\phi$,
\begin{equation}
 \int_\phi^{+\infty}|y(\psi)|^2\,\dd\psi
 \le C_E|y(\phi)|^2,
 \label{eq:ytail}
\end{equation}
and consequently
\begin{equation}
 \int_\phi^{+\infty}|y(\psi)|^3\,\dd\psi
 =o(|y(\phi)|^2).
 \label{eq:cubictail}
\end{equation}
\end{lemma}

\begin{proof}
Since $L_E$ is stable, choose $M>0$ solving
$L_E\T M+ML_E=-I$.  For
$\dot y=L_Ey+\Ocal(|y|^2)$, the quadratic function
$\mathcal V=y\T My$ satisfies $\dot{\mathcal V}\le-\frac12|y|^2$ near the
origin.  Integration to $+\infty$ gives \eqref{eq:ytail}; multiplying by
$\sup_{\psi\ge\phi}|y(\psi)|=o(1)$ gives \eqref{eq:cubictail}.
\end{proof}

\subsection{Quadratic correction on the compatible stable manifold}

Set
\begin{equation}
 A_E:=S\T A_n(U_0)S,
 \qquad
 H_E:=S\T H(U_0)S,
 \qquad
 Q_E:=S\T QS.
 \label{eq:restricted}
\end{equation}
Since $H(U_0)>0$, $H_E>0$.  Because $L_E$ is stable, the Lyapunov equation
\begin{equation}
 L_E\T P+PL_E=Q_E
 \label{eq:Lyapunov}
\end{equation}
has the unique symmetric solution
\begin{equation}
 P=-\int_0^{+\infty}e^{L_E\T t}Q_Ee^{L_Et}\,\dd t.
 \label{eq:Pformula}
\end{equation}

\begin{proposition}[Coordinate invariance]
\label{prop:basis}
The generalized spectrum of
\begin{equation}
 A_E+P-\lambda H_E
 \label{eq:stablepencil}
\end{equation}
is independent of the chosen basis matrix $S$ of $E$.
\end{proposition}

\begin{proof}
If $\widetilde S=ST$, $T\in GL(d,\mathbb R)$, then
$\widetilde L_E=T^{-1}L_ET$ and
$(\widetilde A_E,\widetilde H_E,\widetilde Q_E)=
(T\T A_ET,T\T H_ET,T\T Q_ET)$.  Hence
$\widetilde P=T\T PT$ and the new pencil is congruent to
\eqref{eq:stablepencil}.
\end{proof}

\begin{theorem}[Rank-deficient speed condition]
\label{thm:degenerate}
Let Assumptions~\ref{ass:RET} and \ref{ass:equilibrium} hold together with \eqref{eq:degenerate}, with
$H(U_0)>0$ and $B_s$ invertible.  Suppose that $\mathcal N_s$ in
\eqref{eq:Ns} is a regular invariant $C^2$ manifold near $U_0$, that $U_0$ is
hyperbolic for the restricted flow, and that the stable space $E$ in
\eqref{eq:Estable} has dimension $d\ge1$.  Let $P$ solve
\eqref{eq:Lyapunov}.  Then every nontrivial admissible profile approaching
$U_0$ satisfies
\begin{equation}
 \rho(\phi)=\frac12y(\phi)\T P y(\phi)+o(|y(\phi)|^2),
 \label{eq:rholyap}
\end{equation}
and the finite-amplitude entropy inequality reduces to
\begin{equation}
 \mathcal E
 =\frac12y\T[A_E+P-sH_E]y+o(|y|^2)\ge0.
 \label{eq:lyapcollapse}
\end{equation}
Thus, on the compatible stable manifold, the classical Boillat--Ruggeri
quadratic form $A_n-sH$ is replaced by the reduced quadratic form
$A_E+P-sH_E$.  Accordingly, the relevant local threshold is determined by the
corrected pencil restricted to the stable directions.  Denote by
$\Lambda_{\max}(s;E)$ the largest generalized eigenvalue of this pencil, that
is, the largest root $\lambda$ of
\begin{equation}
 \det(A_E+P-\lambda H_E)=0.
 \label{eq:LambdaEdef}
\end{equation}
Then a necessary condition for a nontrivial smooth profile is
\begin{equation}
 \boxed{s\le\Lambda_{\max}(s;E).}
 \label{eq:boundC}
\end{equation}
\end{theorem}

\begin{proof}
By \eqref{eq:Qdef}, \eqref{eq:stablegraph} and \eqref{eq:VDqdeg},
\[
 V_0\cdot q(U)=\frac12y\T Q_Ey+\Ocal(|y|^3).
\]
For $W(y)=\frac12y\T Py$, equations \eqref{eq:stableflow} and
\eqref{eq:Lyapunov} give
$\dot W=\frac12y\T Q_Ey+\Ocal(|y|^3)$.  Lemma~\ref{lem:ytail} permits
integration from $\phi$ to $+\infty$, yielding \eqref{eq:rholyap}.  Inserting
this in \eqref{eq:profilemaster} gives \eqref{eq:lyapcollapse}.  If
$s>\Lambda_{\max}(s;E)$, the leading quadratic form is negative definite on
$E$, contradicting $\mathcal E\ge0$ on a nontrivial tail.
\end{proof}

Consequently, if $s>\Lambda_{\max}(s;E)$, no nontrivial smooth admissible tail
can approach $U_0$.  As in the Boillat--Ruggeri argument, a travelling shock
structure connecting to that equilibrium, if it exists, must therefore lose
smoothness before reaching the equilibrium; this local condition alone neither
guarantees nor locates a sub-shock.  When $P=0$ the correction disappears; in
the propagated $2\times2$ case treated below, the resulting condition reduces
exactly to the classical Boillat--Ruggeri bound.

\begin{corollary}[Simple real stable mode]
\label{cor:simplemode}
Suppose the asymptotic profile is governed by
\begin{equation}
 X(\phi)=e^{-\kappa\phi}r+o(e^{-\kappa\phi}),
 \qquad \kappa>0,
 \label{eq:dominant}
\end{equation}
where $r\ne0$ is tangent to the conservation leaf and satisfies the linearized
balance and constraint relations
\begin{align}
 -\kappa[A_n(U_0)-sH(U_0)]r
 &= [f_{,U}(U_0)+V_{,U}(U_0)\T q_{,U}(U_0)]r,
 \label{eq:simplebalance}\\
 -\kappa R_{n,U}(U_0)r&=q_{,U}(U_0)r,
 \label{eq:simpleconstraint}\\
 \mathcal J_{s,U}(U_0)r&=0.
 \label{eq:simpleconservation}
\end{align}
Then
\begin{equation}
 s\le
 \frac{r\T\widehat A_n(U_0;\kappa)r}{r\T H(U_0)r},
 \qquad
 \widehat A_n(U_0;\kappa)
 :=A_n(U_0)-\frac{1}{2\kappa}V_0\cdot q_{,UU}(U_0).
 \label{eq:directionalbound}
\end{equation}
\end{corollary}

\begin{proof}
Substitution of \eqref{eq:dominant} into \eqref{eq:rhoexact} gives
$\rho=-(4\kappa)^{-1}V_0\cdot q_{,UU}(U_0)[X,X]+o(|X|^2)$.  The profile
expansion \eqref{eq:profilemaster} and $\mathcal E\ge0$ then give
\eqref{eq:directionalbound}.
\end{proof}

\begin{remark}[Local nature of the nonexistence condition]
\label{rem:localcondition}
The Taylor expansion and stable-manifold argument are used only near the
unperturbed equilibrium $U_0$.  Any global heteroclinic profile approaching
$U_0$ eventually enters that neighborhood.  Thus the derivation is local but
the nonexistence conclusion excludes a global nontrivial $C^1$ profile with
the prescribed state ahead of the shock.  Conversely, the inequality is only
a necessary condition; it is not a global existence theorem and does not by
itself locate a possible sub-shock.
\end{remark}

\section{Propagated rank-deficient constraints in two-field balance laws}
\label{sec:2x2propagated}

The rank-deficient result of Section~\ref{sec:degenerate} is a theorem about
travelling profiles satisfying the differential constraint.  It does not by
itself require that the constraint be propagated by arbitrary solutions of the
full PDE.  We now impose this stronger property and show that, in the minimal
$2\times2$ case with one conservative and one dissipative equation, it has a decisive consequence.

We work in one space dimension and write the original balance laws, before the
constraint is used in the symmetric reduction, in quasilinear form
\begin{align}
 \begin{split}
    & U_t+C(U)U_x=S(U),\\
& C(U):=F_{,U}(U)^{-1}F^1_{,U}(U),
 \quad
 S(U):=F_{,U}(U)^{-1}f(U).
 \label{eq:orig-quasilinear}
 \end{split}
\end{align}
Since $F_{,U}=H>0$, this representation is locally equivalent to the balance
laws.  Consider one scalar differential constraint
\begin{equation}
 R_{,U}(U)U_x=q(U),
 \qquad
 \mathcal C:=R_{,U}(U)U_x-q(U).
 \label{eq:scalar-defect}
\end{equation}
We say that it is \emph{strongly propagated} near $U_0$ if there exist smooth
functions $a(U)$ and $b(U,U_x)$ such that
\begin{equation}
 \mathcal C_t+a(U)\mathcal C_x=b(U,U_x)\mathcal C
 \label{eq:strongprop}
\end{equation}
for every smooth solution of the balance laws.  The precise scalar form of
\eqref{eq:strongprop} is not essential; it is a convenient local expression of
involutivity and is satisfied by the explicit example of the next section.

Let
\begin{equation}
 r:=R_{,U}(U_0),
 \qquad
 B:=S_{,U}(U_0)=H(U_0)^{-1}f_{,U}(U_0),
 \qquad
 C_0:=C(U_0),
 \label{eq:rBC}
\end{equation}
and, in addition, assume
\begin{equation}
 q_{,U}(U_0)=0,
 \qquad
 r\ne0.
 \label{eq:2x2rankdef}
\end{equation}
Thus the scalar constraint is rank deficient at equilibrium.  Put
$K:=\ker r$, which is one-dimensional.

\begin{lemma}[Linear consequences of propagation]
\label{lem:prop-linear}
Under \eqref{eq:strongprop}--\eqref{eq:2x2rankdef}, the line $K$ is invariant
under both $C_0$ and $B$.  If $e$ spans $K$, then
\begin{equation}
 C_0e=\lambda e,
 \qquad
 Be=\beta e,
 \qquad
 rB=\mu r,
 \qquad
 \mu:=b(U_0,0).
 \label{eq:lambdabetamu}
\end{equation}
\end{lemma}

\begin{proof}
Linearize \eqref{eq:strongprop} at the constant equilibrium.  Since
$q_{,U}(U_0)=0$, the linearized defect is $rX_x$.  Comparing separately the
coefficients of $X_{xx}$ and $X_x$ gives
\[
 rC_0=a(U_0)r,
 \qquad
 rB=\mu r.
\]
Hence $C_0K\subset K$ and $BK\subset K$.  Since $K$ is one-dimensional,
\eqref{eq:lambdabetamu} follows.
\end{proof}

The second-order consequence of propagation is equally simple.

\begin{lemma}[Quadratic propagation identity]
\label{lem:prop-quadratic}
Let
\begin{equation}
 Q:=q_{,UU}(U_0).
 \label{eq:Qscalar}
\end{equation}
Then
\begin{equation}
 (2\beta-\mu)Q[e,e]=0.
 \label{eq:resonance2x2}
\end{equation}
\end{lemma}

\begin{proof}
Evaluate \eqref{eq:strongprop} on spatially homogeneous solutions.  Then
$\mathcal C=-q(U)$ and one obtains
\begin{equation}
 q_{,U}(U)S(U)=b(U,0)q(U).
 \label{eq:qSidentity}
\end{equation}
For $U=U_0+\varepsilon e$,
\[
 q(U)=\frac{\varepsilon^2}{2}Q[e,e]+\Ocal(\varepsilon^3),
 \qquad
 q_{,U}(U)S(U)=\varepsilon^2Q[e,Be]+\Ocal(\varepsilon^3).
\]
Using $Be=\beta e$ and comparing the quadratic terms in
\eqref{eq:qSidentity} gives
$\beta Q[e,e]=\mu Q[e,e]/2$, which is \eqref{eq:resonance2x2}.
\end{proof}

We now use the conservative--dissipative block structure.  In the two-field case with one conservation law
and one genuine balance law, a constant nonsingular recombination of the field
equations sends the production to $(0,g)^{\mathsf T}$.  Therefore
\begin{equation}
 \operatorname{rank}f_{,U}(U_0)\le1,
 \qquad
 \operatorname{rank}B\le1.
 \label{eq:rankB}
\end{equation}
This statement is invariant under the constant recombination and is therefore
not tied to a particular representation of the two equations.  We now use also
the relation $e\in\ker r$.  If $\operatorname{rank}B=0$, then $B=0$ and
there is nothing to prove.  If $\operatorname{rank}B=1$, write
$B=a\otimes\ell$ with $a\ne0$ and $\ell\ne0$.  Suppose, for contradiction,
that $\beta\ne0$ and $\mu\ne0$.  From $Be=\beta e$ it follows that $e$ is
parallel to $a$, whereas from
\[
 rB=(r\cdot a)\ell=\mu r
\]
it follows that $r$ is parallel to $\ell$.  Since $re=0$, one then has
$\ell(e)=0$, and therefore $Be=a\,\ell(e)=0$, contradicting
$\beta\ne0$.  Hence
\begin{equation}
 \beta\mu=0.
 \label{eq:betamu0}
\end{equation}

\begin{theorem}[Two-field rigidity for propagated rank-deficient constraints]
\label{thm:2x2rigidity}
Consider a $2\times2$ system satisfying Assumption~\ref{ass:RET}, with one
conservation law and one genuine balance law, together with one strongly propagated scalar differential
constraint.  Assume \eqref{eq:2x2rankdef}.  Let a nontrivial travelling profile
of speed $s$ approach $U_0$ through a regular noncharacteristic one-dimensional
stable tail tangent to $K=\ker R_{,U}(U_0)$, with nonzero spatial decay rate.
Then
\begin{equation}
 Q[e,e]=0.
 \label{eq:Qe0}
\end{equation}
Consequently the quadratic correction $P$ determined by \eqref{eq:Lyapunov}
vanishes on the stable space,
\begin{equation}
 Q_E=0,
 \qquad
 P=0,
 \label{eq:P0}
\end{equation}
and the local necessary condition reduces to the classical form
\begin{equation}
 \boxed{s\le\lambda_{\max}(U_0).}
 \label{eq:2x2classicalbound}
\end{equation}
\end{theorem}

\begin{proof}
Because the linearized constraint gives $r\dot X=0$, the regular tail is tangent
to $K$.  Let $e$ span this line.  The linearized travelling equation in the
original balance variables is
\[
 (C_0-sI)\dot X=BX.
\]
By Lemma~\ref{lem:prop-linear}, $C_0e=\lambda e$ and $Be=\beta e$; hence a mode
$X(\phi)=e^{\kappa\phi}e$ satisfies
\begin{equation}
 \kappa(\lambda-s)=\beta.
 \label{eq:kappabeta}
\end{equation}
The tail is noncharacteristic and $\kappa\ne0$, so $\beta\ne0$.  Equation
\eqref{eq:betamu0} therefore gives $\mu=0$.  The quadratic propagation identity
\eqref{eq:resonance2x2} now yields $Q[e,e]=0$.

For the rank-deficient entropy correction of Section~\ref{sec:degenerate}, the
stable space is one-dimensional and generated by the same compatible direction
$e$.  Hence $Q_E=V_0Q[e,e]=0$.  The Lyapunov equation has the unique stable
solution $P=0$.  Finally, on $K$ the original and constrained symmetric flux
Jacobians have the same action, because their difference is proportional to
$R_{,U}(U_0)$ and therefore annihilates $e$.  The quadratic entropy condition
is consequently the classical one on the compatible direction, and in
particular implies \eqref{eq:2x2classicalbound}.
\end{proof}

\begin{remark}[What the theorem does and does not say]
Theorem~\ref{thm:2x2rigidity} does not invalidate the general rank-deficient reduction
of Theorem~\ref{thm:degenerate}.  That theorem concerns travelling-wave
compatibility and is valid in any dimension under its stated regularity
hypotheses.  The new result says that, in the minimal two-field conservative--dissipative case, the
additional physical requirement of full PDE propagation eliminates a nonzero
quadratic correction whenever the compatible tail has a nonzero linear decay
rate.  A nonzero correction may still be possible in higher dimension, where
$B$ can have a larger dissipative block and the scalar identity
\eqref{eq:betamu0} no longer forces the same conclusion.  This remains open.
\end{remark}

\begin{remark}[Degenerate tails]
If $\beta=0$, equation \eqref{eq:kappabeta} no longer produces a nonzero
exponential spatial rate.  The equilibrium of the travelling ODE is then
degenerate in the relevant direction, although the PDE itself remains
symmetric hyperbolic.  Such algebraic or center-type approaches lie outside the
regular stable-tail analysis developed above and must be analysed directly.  The next example is
of exactly this kind.
\end{remark}

\section{A propagated two-field model: speed selection and algebraic tail}
\label{sec:propagatedexample}

We now give an exact model satisfying the structural requirements used in
Theorem~\ref{thm:2x2rigidity}, but lying in its degenerate-tail alternative.
The construction is inverse in the spirit of Mentrelli and Ruggeri
\cite{MentrelliRuggeri2006}.  We choose the main field $U=(u,v)\T$ and the
potentials
\begin{equation}
 h'=\frac12(u^2+v^2),
 \qquad
 h^{\prime1}=\frac1{12}u^3+\frac12u^2+v^2.
 \label{eq:ss-potentials}
\end{equation}
Let
\begin{equation}
 R(u,v)=v,
 \qquad
 V(u,v)=-\frac54u-\frac14,
 \qquad
 q(v)=v^2(1+v)^2,
 \label{eq:ss-RVq}
\end{equation}
so that
\begin{equation}
 F(U)=\binom uv,
 \qquad
 F^1(U)=\binom{\frac14u^2+u+\frac54v}{2v}.
 \label{eq:ss-flux}
\end{equation}
Using the differential constraint in the symmetric reduction gives
\begin{equation}
 H=I,
 \qquad
 A=\begin{pmatrix}1+\frac12u&0\\0&2\end{pmatrix},
 \qquad
 \lambda_1(U)=1+\frac12u,
 \qquad
 \lambda_2(U)=2.
 \label{eq:ss-charpoly}
\end{equation}

We impose the production
\begin{equation}
 f(U)=q(v)\binom11.
 \label{eq:ss-source}
\end{equation}
The difference of the two field equations is then an exact conservation law,
\begin{equation}
 (u-v)_t+\left(\frac14u^2+u-\frac34v\right)_x=0.
 \label{eq:ss-conservation}
\end{equation}
Equivalently, with
\begin{equation}
 T=\begin{pmatrix}1&-1\\0&1\end{pmatrix},
 \label{eq:ss-T}
\end{equation}
one has
\begin{equation}
 Tf(U)=\binom0{q(v)},
 \qquad
 \widetilde U=T^{-\mathsf T}U=\binom{u}{u+v}.
 \label{eq:ss-RETtransform}
\end{equation}
Thus the nonequilibrium component of the transformed main field is
$\widetilde w=u+v$.  The two equilibrium states used below satisfy
\begin{equation}
 U_1=(1,-1),\qquad U_0=(0,0),
 \qquad
 \widetilde U_1=(1,0),\qquad \widetilde U_0=(0,0).
 \label{eq:ss-equilibrium-normalization}
\end{equation}
In particular, the equilibrium normalization of
Assumption~\ref{ass:equilibrium} is satisfied not only at the reference state
$U_0$ but also at the equilibrium state $U_1$ on the other side of the shock.

The complete constrained system is
\begin{align}
 u_t+\left(\frac14u^2+u+\frac54v\right)_x&=q(v),
 \label{eq:ss-system-u}\\
 v_t+2v_x&=q(v),
 \label{eq:ss-system-v}\\
 v_x&=q(v).
 \label{eq:ss-constraint}
\end{align}
The constraint is genuinely propagated.  Indeed, for
$\mathcal C=v_x-q(v)$,
\begin{equation}
 \mathcal C_t+2\mathcal C_x=q'(v)\mathcal C.
 \label{eq:ss-propagation}
\end{equation}
Hence $\mathcal C=0$ is preserved by the full PDE evolution.

The entropy pair generated by the potentials is
\begin{equation}
 h=\frac12(u^2+v^2),
 \qquad
 h^1=\frac16u^3+\frac12u^2+v^2-\frac14v,
 \label{eq:ss-entropy}
\end{equation}
and
\begin{equation}
 dh^1-U\cdot dF^1
 =-\left(\frac54u+\frac14\right)dv
 =V\,dR.
 \label{eq:ss-genuine}
\end{equation}
On the constitutive domain
\begin{equation}
 \mathcal D=\{(u,v):-1\le u\le1,\ -1\le v\le0\},
 \label{eq:ss-domain}
\end{equation}
the entropy production is
\begin{equation}
 \Sigma(U)=U\cdot f(U)+V(U)q(v)
 =q(v)\left(v-\frac14u-\frac14\right)\le0.
 \label{eq:ss-Sigma}
\end{equation}
Thus the entropy inequality is a constitutive property on the whole domain,
not merely along the travelling orbit.  At the reference equilibrium
$U_0=(0,0)$,
\begin{equation}
 q_{,U}(U_0)=0,
 \qquad
 V_0=-\frac14\ne0,
 \label{eq:ss-degenerate}
\end{equation}
so the constraint is genuinely rank deficient there.

Let $\phi=x-st$ and suppose that the travelling profile is nonconstant in
$v$.  Equations \eqref{eq:ss-system-v} and \eqref{eq:ss-constraint} give
\begin{equation}
 (2-s)\dot v=q(v),
 \qquad
 \dot v=q(v),
 \end{equation}
whence the compatibility condition selects the unique speed
\begin{equation}
 s=1.
 \label{eq:ss-speed}
\end{equation}
The first balance law then reduces to
\begin{equation}
 \frac12u\dot u=-\frac14q(v),
 \end{equation}
and, since $\dot v=q(v)$,
\begin{equation}
 u\frac{du}{dv}=-\frac12,
 \qquad
 u^2+v=C.
 \label{eq:ss-firstintegral}
\end{equation}
For the Rankine--Hugoniot leaf $C=0$, the upper branch
\begin{equation}
 u(v)=\sqrt{-v},\qquad -1\le v\le0
 \label{eq:ss-smoothbranch}
\end{equation}
provides a completely smooth heteroclinic connection
\begin{equation}
 U_1=(1,-1)\longrightarrow U_0=(0,0).
 \label{eq:ss-endstates}
\end{equation}
The characteristic level $\lambda_1=s$ is $u=0$ and is reached by the smooth
branch only at the equilibrium endpoint $U_0$.  Hence no internal sub-shock is
forced.

The approach to $U_0$ is algebraic.  Setting $\eta=-v$ gives
\begin{equation}
 \dot\eta=-\eta^2(1-\eta)^2=-\eta^2+\Ocal(\eta^3),
 \label{eq:ss-eta}
\end{equation}
so that
\begin{equation}
 -v(\phi)\sim\frac1\phi,
 \qquad
 u(\phi)=\sqrt{-v(\phi)}\sim\frac1{\sqrt\phi},
 \qquad \phi\to+\infty.
 \label{eq:ss-algebraic}
\end{equation}
Thus the compatible spatial linear rate vanishes at the reference equilibrium,
as required by the degenerate-tail alternative of
Theorem~\ref{thm:2x2rigidity}.

The same Rankine--Hugoniot leaf also admits entropy-admissible composite
profiles.  For example, at
\begin{equation}
 v_*=-\frac14,
 \qquad
 U_*^- =\left(\frac12,-\frac14\right),
 \qquad
 U_*^+ =\left(-\frac12,-\frac14\right),
 \label{eq:ss-jumpstates}
\end{equation}
the Rankine--Hugoniot condition for \eqref{eq:ss-conservation} holds at
$s=1$.  Moreover,
\begin{equation}
 \lambda_1(U_*^+)=\frac34<s=1<\frac54=\lambda_1(U_*^-),
 \qquad
 s<\lambda_2=2,
 \label{eq:ss-Lax}
\end{equation}
and
\begin{equation}
 [h^1]-s[h]= -\frac1{24}<0.
 \label{eq:ss-entropyjump}
\end{equation}
Hence the jump is a strictly entropy-admissible Lax sub-shock.  After the jump,
the lower branch $u=-\sqrt{-v}$ provides a smooth continuation from $U_*^+$
to the same equilibrium state $U_0$.

\begin{figure}[H]
 \centering
 \includegraphics[width=0.72\textwidth]{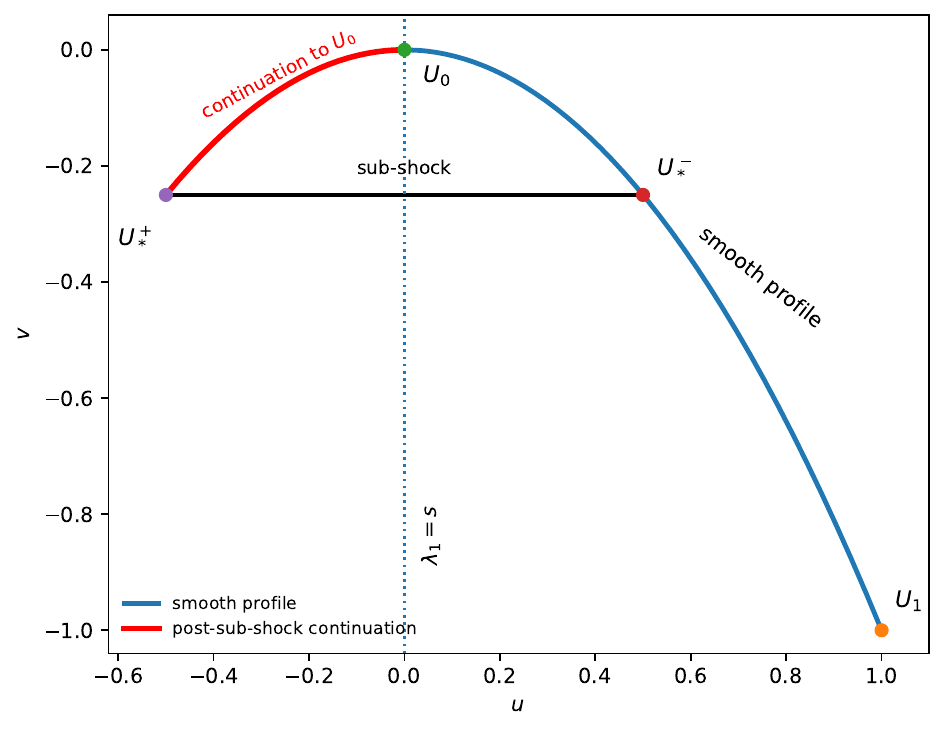}
 \caption{Phase portrait on the Rankine--Hugoniot leaf $u^2+v=0$.  The upper branch gives the completely smooth connection from $U_1=(1,-1)$ to $U_0=(0,0)$.  The composite profile follows the upper branch to $U_*^-=(1/2,-1/4)$, undergoes the entropy-admissible Lax sub-shock to $U_*^+=(-1/2,-1/4)$, and then follows the red lower branch continuously to $U_0$.  The characteristic level $\lambda_1=s$ is $u=0$ and is reached by the smooth branch only at $U_0$.}
 \label{fig:ss-phase}
\end{figure}

\subsection{Entropy uniqueness and selection of the composite profiles}
\label{sec:ss-selection}

The apparent coexistence of a smooth profile and a profile containing a
sub-shock does not imply nonuniqueness of the evolutionary problem.  Indeed,
the constraint and the second balance law give
\begin{equation}
 v_t+v_x=0.
 \label{eq:ss-vtransport}
\end{equation}
Hence, in the moving coordinate $y=x-t$,
\begin{equation}
 v(y,t)=\bar v(y),
 \qquad
 \bar v'=q(\bar v).
 \label{eq:ss-vbar}
\end{equation}
Using the constraint in the first balance law then yields the scalar equation
\begin{equation}
 u_t+\partial_y\left(\frac14u^2\right)
 =-\frac14q(\bar v(y))
 =-\frac14\bar v_y(y).
 \label{eq:ss-scalar}
\end{equation}
Thus the constrained Cauchy problem reduces exactly to a scalar balance law
with a prescribed smooth source.  For each bounded initial datum in the usual
entropy class, the scalar entropy solution is unique \cite{Kruzhkov1970}; for two solutions with
the same prescribed $\bar v$, the source cancels in the standard contraction
argument, so that, whenever the difference is integrable,
\begin{equation}
 \|u_1(t)-u_2(t)\|_{L^1}
 \le
 \|u_{1,\mathrm{in}}-u_{2,\mathrm{in}}\|_{L^1}.
 \label{eq:ss-L1contraction}
\end{equation}
Consequently the nonuniqueness concerns only the travelling-wave boundary-value
problem with prescribed end states, not the entropy Cauchy problem.

A stationary entropy solution of \eqref{eq:ss-scalar} satisfies
\begin{equation}
 \frac14u^2+\frac14\bar v=\text{constant}.
\end{equation}
For the end states \eqref{eq:ss-endstates} the constant is zero.  Defining
\begin{equation}
 A(y):=\sqrt{-\bar v(y)},
 \label{eq:ss-A}
\end{equation}
we therefore obtain the smooth stationary profile
\begin{equation}
 U_\infty(y)=A(y)
 \label{eq:ss-Winf}
\end{equation}
and, for every $a\in\mathbb R$, the composite stationary profile
\begin{equation}
 U_a(y)=
 \begin{cases}
 A(y),&y<a,\\[1mm]
 -A(y),&y>a.
 \end{cases}
 \label{eq:ss-Wa}
\end{equation}
At $y=a$, if $v_*=\bar v(a)$ and
$\alpha_*=\sqrt{-v_*}$, then
\begin{equation}
 u^-=\alpha_*>0>u^+=-\alpha_*.
 \label{eq:ss-generalLax}
\end{equation}
The stationary Rankine--Hugoniot condition is automatic because
$(u^-)^2=(u^+)^2$, while the scalar characteristic speed $u/2$ changes from
positive to negative across the jump.  In the original variables the entropy
jump is
\begin{equation}
 [h^1]-[h]
 =-\frac13\alpha_*^3<0,
 \label{eq:ss-generalentropyjump}
\end{equation}
so every member of \eqref{eq:ss-Wa} is strictly entropy admissible.

Thus the end states and the speed do not determine a unique travelling
structure.  By contrast, a prescribed compatible initial datum together with
the entropy condition determines a unique evolution.  The smooth profile and
the one-parameter family of composite profiles are distinct stationary initial
data, not different solutions issued from the same Cauchy datum.  The remaining
selection question is asymptotic: which stationary member is approached by a
given class of nearby initial data.

\paragraph{Numerical illustration.}
To illustrate the distinction without modifying the first-order model, we
consider the entropy Cauchy problem \eqref{eq:ss-scalar} in the moving
coordinate $\phi=y=x-t$.  Both computations start from \emph{smooth} compatible
initial data with the same asymptotic equilibrium states $U_1=(1,-1)$ and
$U_0=(0,0)$, but the two data are deliberately chosen to produce visibly
different evolutions.  With $A(\phi)=\sqrt{-\bar v(\phi)}$ and
$\bar v(0)=-1/4$, the first datum is a smooth deformation of the upper
stationary branch,
\begin{equation}
 u_{\rm in}^{(a)}(\phi)
 =A(\phi)\left[1-\delta\exp\!\left(-\frac{(\phi+4)^2}{16}\right)\right],
 \qquad \delta=0.22.
 \label{eq:ss-num-smoothIC}
\end{equation}
The second datum is also completely smooth, but is compressive and connects
the upper and lower branches,
\begin{equation}
 u_{\rm in}^{(b)}(\phi)=-A(\phi)\tanh\!\left(\frac{\phi}{2}\right).
 \label{eq:ss-num-shockIC}
\end{equation}
Thus no discontinuity is inserted at $t=0$ in either computation.
The numerical solution is obtained with a first-order monotone finite-volume
Rusanov scheme \cite{Rusanov1961}, using $\Delta\phi=0.005$ and CFL number
$0.45$.  The numerical viscosity is only that of the discretization and is not
an additional constitutive regularization.

\begin{figure}[H]
 \centering
 \includegraphics[width=0.96\textwidth]{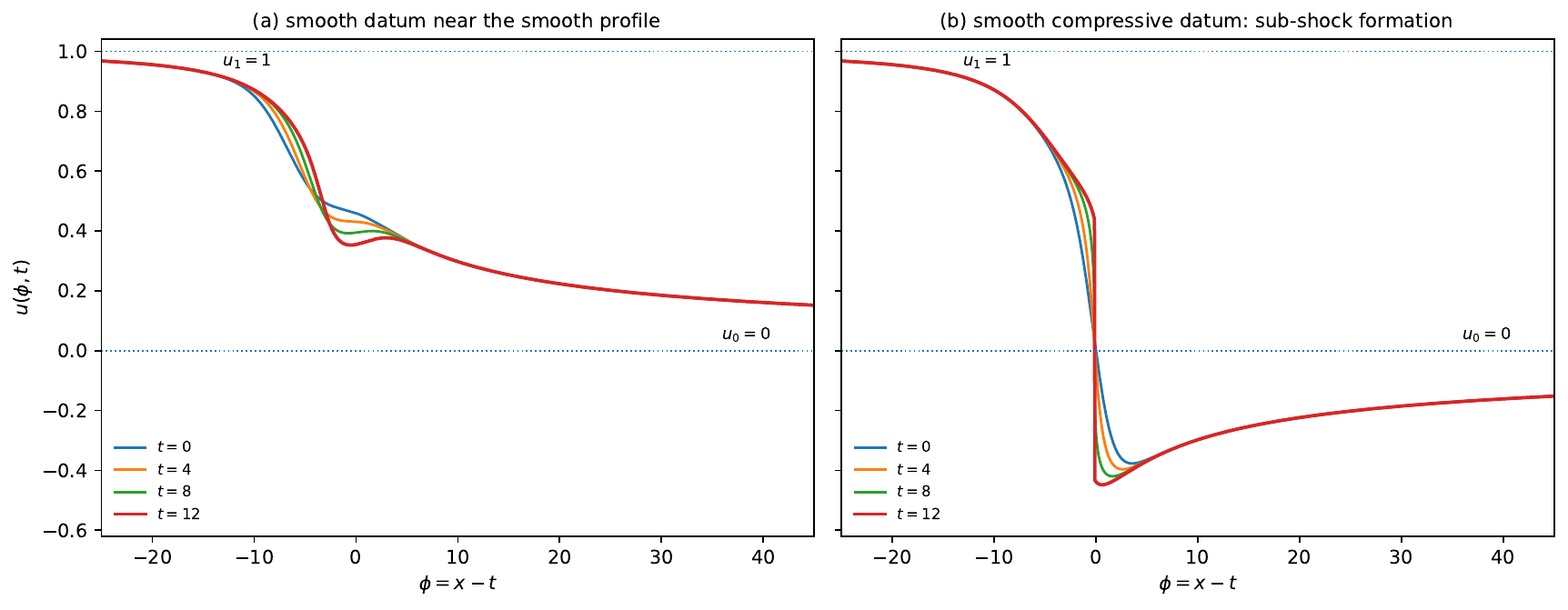}
 \caption{Evolution of $u(\phi,t)$ for two smooth compatible Cauchy data with the same asymptotic equilibrium states.  Panel (a) starts from the smooth deformation \eqref{eq:ss-num-smoothIC}; the solution remains smooth over the displayed time interval.  Panel (b) starts from the smooth compressive datum \eqref{eq:ss-num-shockIC}; the compression steepens and develops a sharp entropy front, numerically representing the formation of a sub-shock.  Curves are shown at $t=0,4,8,12$.}
 \label{fig:ss-numerics}
\end{figure}

The second computation is particularly useful for the interpretation of the
travelling solutions: a sub-shock need not be inserted in the initial datum.
A smooth compatible Cauchy datum may generate it dynamically.  This does not
contradict uniqueness of the entropy Cauchy problem; rather, the initial datum
selects a unique evolution, which may remain smooth or may develop a shock.

\paragraph{Physical interpretation of the bounds.}
The quantity $\lambda_{\max}$ in Theorems~\ref{thm:nondegenerate} and
\ref{thm:homogeneous} is the largest eigenvalue of the constrained symmetric
pencil $A_n-\lambda H$.  For one compatible constraint, Boillat's interlacing
result \cite{Boillat1982} identifies the extreme constrained eigenvalues with
the extreme physical propagation velocities under the usual simple-spectrum
nondegeneracy assumptions.  Theorem~\ref{thm:2x2rigidity} shows that the same
classical local bound survives also for a strongly propagated rank-deficient
scalar constraint in the regular two-field conservative--dissipative regime.

By contrast, the quantity $\Lambda_{\max}(s;E)$ of
Theorem~\ref{thm:degenerate} is a spatial-profile threshold obtained when one
assumes only travelling-wave compatibility.  It remains a valid general
mathematical reduction, but a nonzero quadratic correction $P$ has not
been exhibited here for a genuinely propagated rank-deficient constraint.
The two-field theorem shows why this cannot happen in the regular minimal case.
Whether it can occur for propagated constraints in higher-dimensional systems
with the same conservative--dissipative block structure is an open problem; RET
provides an important physical subclass of such systems.

\section{Conclusions}

The comparison with the unconstrained Boillat--Ruggeri theory is now clear.
The presence of an involutive differential constraint does not in general
replace the largest equilibrium characteristic velocity by another universal
propagation speed.  For homogeneous constraints and for nonhomogeneous
constraints with full-row-rank source Jacobian, the classical necessary
condition survives exactly,
\begin{equation}
 s\le\lambda_{\max}(U_0,n).
 \label{eq:conclusion-classical}
\end{equation}
The longitudinal Gaussian ten-moment plasma with Landau collisions and Gauss's
law gives a physical realization of the full-rank case.

For rank-deficient constraints there are two logically distinct levels.  If one
requires only a travelling profile satisfying the differential constraint, the
conservation laws must first be integrated and the spatial dynamics
restricted to the joint conservation--constraint manifold.  On its stable
tangent space the quadratic contribution of the constraint is described by the
Lyapunov equation
\[
 L_E^{\mathsf T}P+PL_E=Q_E,
\]
and the corresponding necessary condition involves the corrected pencil
$A_E+P-\lambda H_E$.  This correction belongs to the spatial travelling
dynamics; it is not a characteristic velocity of the hyperbolic PDE.

If the constraint is involutive, hence propagated by the full PDE, additional identities appear.  In
the minimal $2\times2$ conservative--dissipative case with one conservation law, one balance law and
one scalar rank-deficient constraint, strong propagation forces the quadratic
constraint term to vanish along every regular noncharacteristic stable tail.
The corresponding quadratic correction is then $P=0$ and the Boillat--Ruggeri bound
is recovered.  Thus the local equilibrium criterion is more robust than the
travelling-wave calculation alone might suggest.

The explicit two-field model of Section~\ref{sec:propagatedexample} shows what
can still be genuinely new.  The source becomes one conservative and one
dissipative equation after a constant nonsingular recombination, the constraint
is exactly propagated by the PDE, and the entropy inequality holds throughout
the constitutive domain.  Compatibility between the balance law and the
constraint selects the unique travelling speed $s=1$.  The smooth
heteroclinic profile approaches the ahead equilibrium only algebraically,
which is precisely the degenerate alternative not covered by the regular
noncharacteristic result of Theorem~\ref{thm:2x2rigidity}.  The same conservation leaf also admits
entropy-admissible Lax composite profiles, but the sub-shock is not forced
because a smooth connection exists between the same end states.  This
coexistence concerns the travelling-wave boundary-value problem and does not
imply nonuniqueness of the evolution: for each compatible initial datum the
reduced scalar entropy problem has a unique solution.  The numerical
experiments further indicate that the distinction is dynamically meaningful:
two different smooth compatible initial data with the same equilibrium end
states can evolve differently, one remaining smooth while the other develops
a compressive sub-shock.

Accordingly, the principal effect of involutive constraints is global rather
than a universal shift of the local characteristic bound.  They restrict the
set of admissible travelling orbits, may select the shock speed, may create
characteristic barriers on individual branches, and may generate several
admissible internal structures even when $s<\lambda_{\max}(U_0)$.  Thus the
end states and propagation speed alone need not determine the internal shock
structure; the initial datum enters the dynamical selection.  In the regular
two-field conservative--dissipative setting, however, the constraints do not defeat the classical
Boillat--Ruggeri upper-speed condition.

The general rank-deficient travelling-wave reduction is therefore complete at the level of local compatibility, whereas the corresponding classification for genuinely propagated constraints remains open beyond the minimal two-field case; this question will be addressed in future work.

The main open question concerns higher-dimensional systems with the
conservative--dissipative block structure and genuinely propagated
rank-deficient constraints.  RET supplies a particularly important physical
class of such systems.  In higher dimension the dissipative block has more
room, the rank-one argument used in Theorem~\ref{thm:2x2rigidity} no longer
applies, and a nonzero quadratic correction determined by the Lyapunov equation may in principle survive.  Other
open regimes are exactly characteristic reductions and degenerate spatial
approaches, which require differential-algebraic or center-manifold methods.

\section*{Acknowledgements}
This work was carried out within the activities of the Gruppo Nazionale per la Fisica Matematica (GNFM) of the Istituto Nazionale di Alta Matematica (INdAM).

\section*{Statements and Declarations}

\textbf{Funding.}
No specific funding was received for this study.

\textbf{Competing interests.}
The authors declare that they have no competing interests.

\textbf{Authors' contributions.}
Both authors contributed equally to all aspects of the work and approved the final manuscript.

\textbf{Data availability.}
No external datasets were used in this study.

\footnotesize
\setlength{\bibsep}{0pt}

\end{document}